\documentclass[12pt,twoside]{amsart}
\usepackage{amssymb}
\usepackage{amscd}
\usepackage{color}

\title[Semiample perturbations for log canonical varieties]
{Semiample perturbations for log canonical varieties 
over an F-finite field containing an infinite perfect field} 
\author{Hiromu Tanaka} 
\address{Department of Mathematics, Imperial College, London, 180 Queen's Gate, 
London SW7 2AZ, UK} 
\email{h.tanaka@imperial.ac.uk}

\newcommand{\Tr}[0]{{\operatorname{Tr}}}

\newcommand{\Spec}[0]{{\operatorname{Spec}}}

\newcommand{\Supp}[0]{{\operatorname{Supp}}}

\newcommand{\Ex}[0]{{\operatorname{Ex}}}
\newtheorem{thm}{Theorem}
\newtheorem{lem}[thm]{Lemma}

\newtheorem{prop}[thm]{Proposition}

\newtheorem{claim}[thm]{Claim}

\theoremstyle{definition}

\newtheorem{dfn}[thm]{Definition}

\newtheorem{rem}[thm]{Remark}
\newtheorem*{ack}{Acknowledgments}

\newcommand{\MO}{\mathcal{O}}
\newcommand{\K}{\mathbb{K}}

\newcommand{\R}{\mathbb{R}}
\newcommand{\Q}{\mathbb{Q}}
\newcommand{\Z}{\mathbb{Z}}
\newcommand{\p}{\mathfrak{p}}

\newcommand{\m}{\mathfrak{m}}

\begin{document}

\maketitle

\begin{abstract}
Let $k$ be an $F$-finite field containing an infinite perfect field 
of positive characteristic. 
Let $(X, \Delta)$ be a projective log canonical pair over $k$. 
In this note we show that, for a semi-ample divisor $D$ on $X$, 
there exists an effective $\mathbb{Q}$-divisor 
$\Delta' \sim_{\mathbb Q} \Delta+D$ such that $(X, \Delta')$ is log canonical 
if there exists a log resolution of $(X, \Delta)$. 
\end{abstract}



\section{Main theorem}

In this note, we prove the following theorem that is non-trivial even for the klt case. 

\begin{thm}\label{Theorem-MMP}
Fix $\K \in \{\Q, \R\}$. 
Let $k$ be an $F$-finite field containing an infinite perfect field $k_0$ of characteristic $p>0$. 
Let $(X, \Delta)$ be a projective log canonical $($resp. klt$)$ 
pair over $k$, 
where $\Delta$ is an effective $\K$-divisor. 
Let $D$ be a semi-ample $\K$-Cartier $\K$-divisor on $X$. 
If there is a log resolution of $(X, \Delta)$, 
then there exists an effective $\K$-Cartier $\K$-divisor $D' \sim_{\K} D$ 
such that $(X, \Delta+D')$ is log canonical $($resp. klt$)$. 
\end{thm}

In characteristic zero, Theorem~\ref{Theorem-MMP} holds by Bertini's theorem 
for free divisors, which fails in positive characteristic. 
Our proof depends on the theory of $F$-singularities. 
More precisely, 
Theorem~\ref{Theorem-MMP} follows from 
Proposition~\ref{Theorem-F1} and Proposition~\ref{Theorem-F}.

\begin{prop}\label{Theorem-F1}
Fix $\K \in \{\Q, \R\}$. 
Let $k$ be an $F$-finite field containing an infinite perfect field $k_0$ of characteristic $p>0$. 
Let $(X, \Delta)$ be a projective strongly $F$-regular pair over $k$, where $\Delta$ is an effective $\K$-divisor. 
Let $D$ be a semi-ample $\K$-divisor on $X$. 
If $X$ is regular, 
then there exists an effective $\K$-divisor $D' \sim_{\K} D$ 
such that $(X, \Delta+D')$ is strongly $F$-regular. 
\end{prop}

\begin{prop}\label{Theorem-F}
Let $k$ be an $F$-finite field containing an infinite perfect field $k_0$ of characteristic $p>0$. 
Let $X$ be a projective regular variety over $k$ and 
let $\Delta$ be an effective simple normal crossing $\Q$-divisor on $X$ 
whose coefficients are contained in $[0, 1]$. 
Let $D$ be a semi-ample $\Q$-divisor on $X$. 
Then there exists an effective $\Q$-divisor $D' \sim_{\Q} D$ 
such that $(X, \Delta+D')$ is sharply $F$-pure. 
\end{prop}

{\bf Idea of Theorem~\ref{Theorem-MMP}:} 
We overview the proof of Theorem~\ref{Theorem-MMP}. 
We only treat the case $\K=\Q$, $(X, \Delta)$ is klt, and $k$ is an infinite perfect field. 

Since there is a log resolution of $(X, \Delta)$, 
we may assume that $X$ is smooth and $\Delta$ is simple normal crossing. 
In this case, the notions of klt and strongly $F$-regular singularities coincide, 
hence it suffices to show Proposition~\ref{Theorem-F1} 
because strongly $F$-regular singularities are klt. 

We may assume that $D$ is base point free. 
By \cite[Theorem~B]{PSZ}, 
we can find a large integer $m \in \Z_{>0}$ such that 
$$\left(X, \Delta+\frac{1}{m}(D_1+\cdots+D_{\dim X})\right)$$ 
is strongly $F$-regular for every $D_i \in |D|$. 
Since $D$ is base point free, we can find members 
$$D_1, \cdots, D_m \in |D|$$
which satisfy the following property:  
\begin{enumerate}
\item[$(*)$]{the intersection $\bigcap_{j \in J}D_j$ is empty for every $|J|=\dim X+1$. }
\end{enumerate}
Set 
$$D':=\frac{1}{m}(D_1+\cdots+D_m) \sim_{\Q} D.$$
By $(*)$, for every point $x \in X$, 
there is an open subset $x \in U \subset X$ such that 
$$(U, (\Delta+D')|_U)=\left(U, \Delta|_U+\frac{1}{m}\sum_{i \in I} (D_i|_U)\right)$$
for some $I \subset \{1, 2, \cdots, m\}$ with $0\leq |I| \leq \dim X$. 
Therefore, $(X, \Delta+D')$ is strongly $F$-regular. 

If $k$ is not a perfect field, 
then we need to replace \cite[Theorem~B]{PSZ} in the above argument with Proposition~\ref{special-general}. 
Although the proofs of Proposition~\ref{Theorem-F1} and Proposition~\ref{Theorem-F} 
are very similar, the proof of Proposition~\ref{Theorem-F} needs the inversion of adjunction for $F$-singularities 
established by \cite{F-adjunction}. 

\medskip

{\bf Assumption on the base field:} 
Let us consider about the assumption that $k$ contains an infinite perfect field $k_0$. 
First, we impose a restrictive condition: $k$ is perfect. 
Then, our assumption $k_0 \subset k$ is equivalent to the condition that $k$ is an infinite perfect field. 
In this case, 
a key result (Proposition~\ref{special-general}) almost follows from \cite[Theorem~B]{PSZ}, 
however we need the arguments 
in the proofs of Proposition~\ref{Theorem-F1} and Proposition~\ref{Theorem-F}. 
For the case when $k$ is a finite field, 
our proof encounters a similar technical difficulty to the Bertini theorem for very ample divisors. 
Although the author does not know whether Theorem~\ref{Theorem-MMP} holds for finite fields, 
the method of \cite{Poonen} may help us. 


Second, let us go back to the general situation. 
In the proof of a key result (Proposition~\ref{special-general}), 
we will make use of rational points or its $p^e$-powers of $\mathbb A_k^r$. 
To assure the existence of infinitely many $p^{\infty}$-torsion points, 
we assume that $k$ contains an infinite perfect field. 

\medskip

{\bf Motivation:} 
Originally the motivation for Theorem~\ref{Theorem-MMP} is 
to show the relative klt/lc abundance theorem assuming the absolute abundance theorem. 
For example, given a morphism $\pi:(X, \Delta) \to S$ from a projective klt variety $(X, \Delta)$ to a normal variety $S$, 
if $K_X+\Delta$ is $f$-nef, then we see that $K_X+\Delta+\pi^*A_S$ is nef for some ample divisor $A_S$ 
up to the cone theorem. 
Since $\pi^*A_S$ is semi-ample, Theorem~\ref{Theorem-MMP} implies that 
there is an effective $\Q$-divisor $\Delta' \sim_{\Q} \Delta+\pi^*A_S$ 
such that $(X, \Delta')$ is klt and $K_X+\Delta$ is nef. 
Thus, we could reduce the problem to the absolute case. 
In this situation, our assumption on the base field $k$ is harmless because 
we can assume this as follows: first take the base change 
to the composite field of $k$ and $\overline{\mathbb F}_p$, 
and second we can find a subfield of $k$ 
which is finitely generated over $\overline{\mathbb F}_p$ and defines all the given varieties and morphisms. 

\medskip 

{\bf Related results:}
Recently, some problems on birational geometry are solved 
by the theory of $F$-singularities (\cite{CHMS}, \cite{CTX}, \cite{HX}, \cite{Mustata}). 
By using it, also this paper shows a result on birational geometry (Theorem~\ref{Theorem-MMP}). 
For some related topics of $F$-singularities, see \cite{BSTZ}, \cite{BST}. 
In particular, Bertini's theorem of very ample divisors holds for $F$-singularities (\cite{SZ}). 

\begin{ack}
The author would like to thank Professors 
Paolo Cascini, 
Mircea Musta\c{t}\u{a}, 
Zsolt Patakfalvi, 
Karl Schwede and 
Shunsuke Takagi 
for very useful comments and discussions. 
He would like to thank the referee 
for reading the manuscript carefully and for suggesting several improvements. 
The author was funded by EPSRC.
\end{ack}

\section{Notation}

We say $X$ is a {\em variety} over a field $k$ (or $k$-variety) if 
$X$ is an integral scheme which is separated and of finite type over $k$. 


Let $X$ be a normal variety over a field $k$. 
For a closed subset $Z$ of $X$, 
we say $Z$ is a {\em simple normal crossing divisor} 
if, for the irreducible decomposition $Z=\bigcup_{i \in I} D_i$, 
we obtain $\dim D_i=\dim X-1$ for every $i \in I$ and 
$\bigcap_{j \in J} D_j$ is regular for every subset $\emptyset \neq J \subset I$, 
where we consider $D_i$ as the reduced scheme and 
the intersection $\bigcap_{j \in J} D_j$ means a scheme-theoretic intersection.

Let $X$ be a normal $k$-variety and 
let $D$ be an effective $\R$-divisor. 
We say $\mu:W \to X$ is a {\em log resolution} of $(X, D)$ 
if $\mu$ is a projective birational morphism, 
$W$ is regular and $\mu^{-1}(D)\cup \Ex(\mu)$ 
is a simple normal crossing divisor.

We will freely use the notation and terminology in \cite{Kollar}. 
In particular, for the definitions of klt and log canonical singularities, 
see \cite[Definition~2.8]{Kollar}. 

For the definitions of strongly $F$-regular and sharply $F$-pure pairs, 
see \cite[Definition~2.7]{F-adjunction} and \cite[Definition~2.7]{CTX}. 
For $e\in \Z_{>0}$, a normal variety $X$ and an effective $\Z$-divisor $D$ on $X$, 
we have {\em a trace map} 
$$\Tr_X^e(D):F_*^e(\MO_X(-(p^e-1)K_X-D)) \to \MO_X.$$
For the definition and the basic properties of trace map, see \cite[Section~2.3]{CTX}.

We say an $\R$-Cartier $\R$-divisor $D$ is {\em semi-ample} 
if we can write $D=\sum_{1 \leq i\leq r} a_iD_i$, where 
$a_i \in \R_{\geq 0}$ and $D_i$ is a semi-ample Cartier divisor. 

For an $\mathbb F_p$-algebra $A$, 
we define $F_*A$ to be $A$ as an abelian group and we equip $F_*A$ 
with the $A$-module structure satisfying $a \cdot F_*x=F_*(a^px)$ 
for all $a, x \in A$, 
where we denote by $F_*y$ the same element as $y \in A$. 
A finite subset $\{x_1, \cdots, x_n\}$ of $A$ is said be a $p$-{\em basis} if 
the $A$-module homomorphism 
\begin{eqnarray*}
\bigoplus_{i=1}^n A &\to& F_*A,\\ 
(a_1, \cdots, a_n) &\mapsto& \sum_{i=1}^n a_i \cdot F_*(x_i)=F_*\left(\sum_{i=1}^n a_i^p x_i\right)
\end{eqnarray*}
is bijective.

\section{Proofs}

We recall a basic lemma. 

\begin{lem}\label{p-basis}
Let $A$ be an integral domain of characteristic $p>0$ and 
assume that $A$ has a $p$-basis $\{x_1, \cdots, x_n\}$. 
Then, $\{x_1, \cdots, x_n, t_1, \cdots, t_r\}$ is a $p$-basis of $A[t_1, \cdots, t_r]$. 
\end{lem}

\begin{proof}
We may assume that $r=1$ and set $t_1=:t$. 
Consider the following $A[t]$-module homomorphism
\begin{eqnarray*}
\varphi:\bigoplus_{0\leq i_{\ell}, j<p} A[t]\cdot x_1^{i_1}\cdots x_n^{i_n}t^j &\to& F_*(A[t]),\\
x_1^{i_1}\cdots x_n^{i_n}t^j &\mapsto& F_*(x_1^{i_1}\cdots x_n^{i_n}t^j),
\end{eqnarray*}
where the left hand side is the free $A[t]$-module generated 
by a basis $\{x_1^{i_1}\cdots x_n^{i_n}t^j\}_{0\leq i_{\ell}, j<p}$. 
The surjectivity is clear. 
We show the injectivity. 
Assume that 
$$\sum_{0\leq i_{\ell}, j<p} f_{I, j}(t)^px_1^{i_1}\cdots x_n^{i_n}t^j=0\,\,\,\,\, {\rm in}\,\,A[t],$$
for some $f_{I, j}(t)\in A[t]$ where $I:=(i_1, \cdots, i_n)$ is the multi-index. 
We show $f_{I, j}(t)=0$ for every $I$ and $j$. 
We can write $f_{I, j}(t)=\sum_k a_k^{I, j}t^k$ for some $a_k^{I, j} \in A$. 
We obtain 
\begin{eqnarray*}
\sum_{I, j} f_{I, j}(t)^px_1^{i_1}\cdots x_n^{i_n}t^j
&=&\sum_{I, j} \left(\sum_k a_k^{I, j}t^k\right)^px_1^{i_1}\cdots x_n^{i_n}t^j\\
&=&\sum_{I, j, k} \left(a_k^{I, j}\right)^px_1^{i_1}\cdots x_n^{i_n}t^{pk+j}\\
&=&0.
\end{eqnarray*}
Since every coefficient of $t^{pk+j}$ vanishes, 
we obtain 
$$\sum_{I} \left(a_k^{I, j}\right)^px_1^{i_1}\cdots x_n^{i_n}=0$$
for every $j$ and $k$. 
By the definition of $p$-basis, we obtain $a_k^{I, j}=0$ for every $I, j$ and $k$. 
Thus $f_{I, j}(t)=\sum_k a_k^{I, j}t^k=0$ for every $I$ and $j$. 
We are done. 
\end{proof}

The following lemma essentially follows from \cite{F-adjunction}.

\begin{lem}\label{adjunction}
Let $Y$ be a regular variety over an $F$-finite field of characteristic $p>0$. 
Let $S$ be a reduced simple normal crossing divisor
and let $S=\sum_{i\in I} S_i$ be the irreducible decomposition. 
Set $T:=\bigcap_{i \in I} S_i$. 
Let $E$ be an effective Cartier divisor on $X$ such that 
any irreducible component of $T$ is not contained in $\Supp E$. 
Fix a $($possibly non-closed$)$ point $y\in T$ and $e\in\mathbb Z_{>0}$. 
Then, the following assertions are equivalent. 
\begin{enumerate}
\item{$\Tr_Y^e((p^e-1)S+E)$ is surjective at $y$. }
\item{$\Tr_{T}^e(E|_T)$ is surjective at $y$. }
\item{The $\MO_{Y, y}$-module homomorphism 
$$\MO_{Y, y} \xrightarrow{F^e} F^e_*\MO_{Y, y} \hookrightarrow F^e_*(\MO_{Y, y}((p^e-1)S+E))$$
splits. 
}
\item{The $\MO_{T, y}$-module homomorphism 
$$\MO_{T, y} \xrightarrow{F^e} F^e_*\MO_{T, y} \hookrightarrow F^e_*(\MO_{T, y}(E|_T))$$
splits. }
\end{enumerate}
\end{lem}

\begin{proof}
It follows from \cite[Proposition 2.6]{CTX} that (1) and (3) (resp. (2) and (4)) are equivalent. 
Thus it suffices to show that (1) and (2) are equivalent. 
By the construction of the trace map, 
we obtain the following commutative diagram (cf. \cite[Lemma~2.6(1)]{trace-map}): 
$$\begin{CD}
F_*^e(\mathcal O_Y(-(p^e-1)(K_Y+S)-E))@>>{\rm surjection}>F_*^e(\mathcal O_T(-(p^e-1)K_{T}-E|_{T}))\\
@VV\Tr_Y^e((p^e-1)S+E)V @VV\Tr_T^e(E|_T)V\\
\mathcal O_Y@>\rho>{\rm surjection}>\mathcal O_{T}.
\end{CD}$$
If $\Tr_Y^e((p^e-1)S+E)$ is surjective at $y$, then 
so is $\Tr_T^e(E|_T)$ by a diagram chase. 
Thus (1) implies (2). 
Assume that $\Tr_T^e(E|_T)$ is surjective at $y$. 
By a diagram chase, 
there exists $a\in{\rm Im}(\Tr_Y^e((p^e-1)S+E))$ such that 
$\rho(a)=1\in\mathcal O_{T, y}.$ 
Then, we see $a-1\in I_T\MO_{Y, y}\subset \mathfrak m_y\mathcal O_{Y, y}.$ 
This implies that $a$ is a unit of $\MO_{Y, y}$. 
Thus (2) implies (1). 
\end{proof}

Since we would like to treat rational points of $\mathbb A_k^r$ and their $p^e$-powers, 
we introduce the following terminologies. 

\begin{dfn}\label{def-coordinate}
Let $k$ be a field. 
A {\em coordinate} $t:=[t_1, \cdots, t_r]$ of $\mathbb A_k^r$ is 
a set of elements of $\MO_{\mathbb A_k^r}(\mathbb A_k^r)$ which satisfies 
$k[t_1, \cdots, t_r]=\MO_{\mathbb A_k^r}(\mathbb A_k^r)$. 
This induces a bijection: 
$$\theta:k^r \to \mathbb A_k^r(k),\,\,\, (c_1, \cdots, c_r) \mapsto (t_1-c_1, \cdots, t_r-c_r).$$
For a subfield $k_0 \subset k$, 
we define $\mathbb A_k^r(t, k_0):=\theta(k_0^r)$. 
\end{dfn}

\begin{rem}\label{translation}
The set $\mathbb A_k^r(t, k_0)$ is a subset of $\mathbb A_k^r(k)$, 
which depends on the choice of a coordinate $t=[t_1, \cdots, t_r]$. 
Actually, for every $a \in \mathbb A_k^r(k)$ and every subfield $k_0 \subset k$, 
we can find a coordinate $t=[t_1, \cdots, t_r]$ such that $a \in \mathbb A_k^r(t, k_0)$. 
On the other hand, we obtain 
$\mathbb A_k^r(t, k_0)=\mathbb A_k^r(t', k_0)$ 
for another coordinate $t'=[t'_1, \cdots, t'_r]$ such that $t'_i=t_i-c_i$ with $c_i \in k_0$. 
\end{rem}



Proposition~\ref{special-general} is a key result in this paper, 
which compares a special fiber and general fibers. 
For the case where $k$ is perfect, Proposition~\ref{special-general} is very similar to \cite[Theorem~B]{PSZ}.

\begin{prop}\label{special-general}
Let $k$ be an $F$-finite field of characteristic $p>0$. 
Let $X$ be a regular variety over $k$. 
We fix the followings. 
\begin{itemize}
\item{$e\in\mathbb Z_{>0}$.}
\item{$T:=\mathbb A^r_k$. }
\item{An effective Cartier divisor $E$ on $X\times_k T$ such that 
$\Supp\,E$ contains no fibers of the second projection $X\times_k T \to T$.}
\item{A closed point $x\in X$ and a rational point $a \in T(k)$. }
\item{A coordinate $t=[t_1, \cdots, t_r]$ of $T$ such that $a \in T(t, k^{p^e})$ 
(cf. Definition~\ref{def-coordinate}).}
\end{itemize}
If $\Tr^e_{X\times_k \{a\}}(E|_{X\times_k\{a\}})$ is surjective at $(x, a)$, 
then there exists an open subset $(x ,a)\in U$ of $X\times_k T$ 
such that
$\Tr^e_{X\times \{b\}}(E|_{X\times_k\{b\}})$ is surjective at $(y, b)$ 
for every closed point $(y, b)\in U$ with a closed point $y \in X$ and $b \in T(t, k^{p^e})$. 
\end{prop}

\begin{proof}
We may assume that 
$$X=\Spec\,R_X,\,\,\,\, T=\Spec\,k[t_1, \cdots, t_r]=\Spec\,R_T.$$
Moreover, we can assume 
that $a=(t_1, \cdots, t_r) \subset R_T$, that is, $a$ is the origin (Remark~\ref{translation}). 
By shrinking $\Spec\,R_X$ around $x$, we may assume that 
$R_X$ has a $p$-basis $\{x_1, \cdots, x_n\}$: 
$$R_X=\bigoplus_{0\leq i_{\alpha} < p^e} R_X^{p^e}x_1^{i_1}\cdots x_n^{i_n}.$$
Take an element $0\neq \rho \in R_X\otimes_k R_T$ which satisfies the following properties.  
\begin{enumerate}
\item{$(x, a)\in \Spec((R_X\otimes_k R_T)[\rho^{-1}])$. }
\item{$E|_{\Spec((R_X\otimes_k R_T)[\rho^{-1}])}={\rm div}(f)$ where $f\in (R_X\otimes R_T)[\rho^{-1}]$.}
\end{enumerate}
Set 
$$S:=(R_X\otimes_k R_T)[\rho^{-1}].$$
Let $S_i$ on $X\times_k T$ be the pull-back of the regular effective Cartier divisor ${\rm div}(t_i)$ on $T$. 
By Lemma~\ref{adjunction}, 
we see that the trace map 
$$\Tr^e_{X\times T}\left((p^e-1)\sum_{i=1}^rS_i+E\right)$$ 
is surjective at $(x, a).$ 
Shrinking $\Spec\,S$ around $(x, a)$ if necessary, 
we obtain the following splitting by Lemma~\ref{adjunction}: 
$${\rm id}_S:S\to F_*^eS\xrightarrow{\times f(t_1\cdots t_r)^{p^e-1}} F_*^eS
\xrightarrow{\varphi}S.$$
Again by Lemma~\ref{adjunction}, it suffices to show the splitting of 
$$S\to F_*^eS\xrightarrow{\times f((t_1-b_1)\cdots (t_r-b_r))^{p^e-1}} F_*^eS$$
for general $(b_1, \cdots, b_r) \in (k^{p^e})^r$. 

Since 
$$\{x_1, \cdots, x_n, t_1-b_1, \cdots, t_r-b_r\}$$ 
is a $p$-basis of $R_X\otimes_k k[t_1, \cdots, t_r]$ for every $b_i\in k$ (Lemma~\ref{p-basis}), 
it is also a $p$-basis of 
$$(R_X\otimes_k k[t_1, \cdots, t_r])[\rho^{-1}]=(R_X\otimes_k R_T)[\rho^{-1}]=S.$$ 
For every $b:=(b_1, \cdots, b_r)\in k^r$, we obtain the following expressions 
\begin{eqnarray*}
f&=&\sum_{0\leq i_{\alpha}<p^e, 0\leq j_{\beta}<p^e} \xi_{I, J}^{p^e}(x_1^{i_1}\cdots x_n^{i_n}t_1^{j_1}\cdots t_r^{j_r})\\
&=&\sum_{0\leq i_{\alpha}<p^e, 0\leq j_{\beta}<p^e} \xi_{I, J}(b)^{p^e}(x_1^{i_1}\cdots x_n^{i_n}(t_1-b_1)^{j_1}\cdots (t_r-b_r)^{j_r})
\end{eqnarray*}
where $\xi_{I, J}$, $\xi_{I, J}(b) \in S$, $I:=(i_1,\cdots, i_n)$, and $J:=(j_1, \cdots, j_r)$. 
Clearly, $\xi_{I, J}(0)=\xi_{I, J}$. 
By applying the binomial expansion, if $b_{\ell} \in k^{p^e}$, then we can write 
$$t_{\ell}^{j_{\ell}}=((t_{\ell}-b_{\ell})+b_{\ell})^{j_{\ell}}=\sum_{\nu} c_{\ell, j_{\ell}, \nu}^{p^e}(t_{\ell}-b_{\ell})^{\nu}$$
for some $c_{\ell, j_{\ell}, \nu} \in k$ such that $c_{\ell, j_{\ell}, 0}=b_{\ell}^{j_{\ell}}.$ 
Therefore, we obtain  
$$\xi_{I, 0}(b)^{p^e}=\sum_{0\leq j_{\beta} < p^e} \xi_{I, J}^{p^e}b_1^{j_1}\cdots b_r^{j_r}$$
for every $I$ and $b=(b_1, \cdots, b_r)\in (k^{p^e})^r$.

\begin{claim}\label{c1}
For any point $y \in \Spec\,S$ and any $k$-rational point $b \in T(k)$, 
if $g$ is the image of $f$ by 
$$S=(R_X \otimes_k R_T)[\rho^{-1}] \to (R_X \otimes_k (R_T/\m_b))[\rho^{-1}] \to \MO_{X \times\{b\}, (b, y)} 
\xrightarrow{\simeq} \MO_{X, y},$$
then the following are equivalent. 
\begin{enumerate}
\item[(i)] $\Tr^e_{X\times \{b\}}(E|_{X\times_k\{b\}})$ is surjective at $(y, b)$. 
\item[(ii)] The $\MO_{X, y}$-module homomorphism 
$$\MO_{X, y} \to F_*^e \MO_{X, y} \xrightarrow{\times g} F^e_*\MO_{X, y}$$
splits. 
\item[(iii)] There exists an index $I=(i_1, \cdots, i_n)$ with $0 \leq i_{\alpha}<p^e$ 
such that $\xi_{I, 0}(b) \not\in \m_{(y, b)}$. 
\end{enumerate}
\end{claim}
\begin{proof}[Proof of Claim \ref{c1}]
Thanks to Lemma~\ref{adjunction}, (i) and (ii) are equivalent. 
Since $\{x_1, \cdots, x_n\}$ is a $p$-basis of $\MO_{X, y}$, we get  
\begin{eqnarray*}
F^e_*\MO_{X, y}&=&\bigoplus_{0 \leq i_{\alpha}<p^e} \MO_{X, y}\cdot x_1^{i_1} \cdots x_n^{i_n}\\
g&=&\bigoplus_{0 \leq i_{\alpha}<p^e}(\xi_{I, 0}(b) \mod \m_b)^{p^e} \cdot x_1^{i_1} \cdots x_n^{i_n}
\end{eqnarray*}
which implies that (ii) is equivalent to (iii). 
This completes the proof of Claim \ref{c1}. 
\end{proof}

Consider the following polynomial 
$$\eta_{I, 0}(z_1, \cdots, z_r):=
\sum_{0\leq j_{\beta}< p^e} \xi_{I, J}^{p^e}z_1^{j_1}\cdots z_r^{j_r} \in S[z_1, \cdots, z_r].$$
For $b=(b_1, \cdots, b_r)\in (k^{p^e})^r$, 
we obtain 
$$\eta_{I, 0}(b_1, \cdots, b_r)=\sum_{0\leq j_{\beta}< p^e} \xi_{I, J}^{p^e}b_1^{j_1}\cdots b_r^{j_r}=\xi_{I, 0}(b)^{p^e}.$$
Recall 
$$t_i \in k[t_1, \cdots, t_r]=R_T \subset R_X\otimes_k R_T \subset (R_X\otimes_k R_T)[\rho^{-1}]=S,$$
where we write $t_i=1\otimes_k t_i \in S$ by abuse of notation. 
Thus, we obtain 
$$\eta_{I, 0}(t_1, \cdots, t_r) \in S.$$
Then, for a closed point $y \in X$ and $b=(b_1, \cdots, b_r) \in (k^{p^e})^r$ 
(i.e. $b$ corresponds to the maximal ideal $(t_1-b_1, \cdots, t_r-b_r)$), 
the following three assertions are equivalent. 
\begin{itemize}
\item{$\eta_{I, 0}(t_1, \cdots, t_r) \not\in \m_{(y, b)}$.}
\item{$\eta_{I, 0}(b_1, \cdots, b_r) \not\in \m_{(y, b)}$.}
\item{$\xi_{I, 0}(b) \not\in \m_{(y, b)}$. }
\end{itemize}

Since $\varphi$ gives the above splitting, 
Claim \ref{c1} enables us to find a multi-index $I'=(i_1',\cdots, i_n')$ such that 
$\xi_{I', 0}=\xi_{I', 0}(0)\not\in\mathfrak m_{(x, a)}$. 
This implies  
$$\eta_{I', 0}(t_1, \cdots, t_r) \not\in \m_{(x, a)}.$$ 
Thus, we can find an open set $(x, a) \in U \subset \Spec\,S$, 
such that 
$$\eta_{I', 0}(t_1, \cdots, t_r) \not \in \p$$
for every prime ideal $\p \in U$. 
Therefore, for every $(y, b)\in U$ where $y \in X$ is a closed point and $b\in T(k^{p^e})$, 
we obtain $\xi_{I', 0}(b) \not\in \m_{(y, b)}$. 
It follows from Claim \ref{c1} that $\Tr^e_{X\times \{b\}}(E|_{X\times_k\{b\}})$ is surjective at such $(y, b)$. 
\end{proof}

\begin{proof}[Proof of Proposition~\ref{Theorem-F1}]
First, we reduce the proof to the case $\K=\Q$. 
By enlarging coefficients of $\Delta$, we may assume that $\Delta$ is a $\Q$-divisor 
(cf. \cite[Remark~2.8(2)]{CTX}). 
We can write $D=\sum_{1\leq i\leq s} a_iD_i$, where $a_i\in\R_{\geq 0}$ and 
each $D_i$ is a semi-ample Cartier divisor. 
By induction on $s$, we may assume that $s=1$. 
Thus we obtain $D=a_1D_1$. 
By replacing $a_1$ with $\ulcorner a_1 \urcorner$, we may assume that $D$ is a semi-ample Cartier divisor. 
Thus, we could reduce the proof to the case $\K=\Q$. 

From now on, we assume that $\K=\Q$ and 
we show the assertion in the proposition. 
\begin{claim}\label{c-2}
There exists an effective $\Q$-divisor $H$ on $X$ 
satisfying the following properties. 
\begin{itemize}
\item{$(X, \Delta+H)$ is strongly $F$-regular.}
\item{$\Supp \Delta \subset \Supp H$.}
\item{$(X\setminus \Supp H, \Delta|_{X\setminus \Supp H}=0)$ is globally $F$-regular.}
\end{itemize}
\end{claim}
\begin{proof}[Proof of Claim \ref{c-2}]
Fix a closed point $x \in X \setminus \Supp \Delta$ around which $X$ is regular. 
There exists an affine open neighborhood $U$ of $x \in X$ such that $(U, 0)$ is globally $F$-regular. 
Take an effective $\Z$-divisor $H_1$ on $X$ such that $\Supp \Delta  \subset \Supp H_1$ and 
that $X \setminus \Supp H_1 \subset U$. 
We can find a small positive rational number $\epsilon$ such that 
for $H:=\epsilon H_1$, the pair $(X, \Delta+H)$ is strongly $F$-regular (cf. \cite[Remark 2.8]{CTX}). 
This completes the proof of Claim \ref{c-2}.
\end{proof}

By enlarging coefficients of $\Delta$ and $H$ a little, 
we may assume that $(p^{d_0}-1)\Delta$ and $(p^{d_0}-1)H$ are $\Z$-divisors for some $d_0 \in \Z_{>0}$. 
By replacing $D$ with its some multiple, 
we may assume that $D$ is a Cartier divisor such that $|D|$ is base point free. 
In particular, $\dim_kH^0(X, D) \geq 2$, otherwise $D \sim 0$ and the assertion is trivial.

Set $T_1:=\mathbb A_k^q=\mathbb A(H^0(X, D))$ (i.e. $q:=\dim_kH^0(X, D)$) and 
$$T:=(T_1)^{\dim X}=\mathbb A_k^q \times_k \cdots \times_k \mathbb A_k^q \simeq \mathbb A_k^{q\dim X}.$$
We can find an effective Cartier divisor $\mathcal D$ on $X \times_k T$ 
which satisfies the following properties: 
\begin{itemize}
\item{Each $k$-rational point $a\in (\mathbb A_k^q(k))^{\dim X} = T(k)$ 
corresponds to a pair $(D_1, \cdots, D_{\dim X})$ 
where, for every $i$, we obtain $D_i \sim D$ is an effective Cartier divisor. 
In this case, we write $a=[D_1, \cdots, D_{\dim X}]$. }
\item{$\mathcal D|_{X\times_k\{[D_1, \cdots, D_{\dim X}]\}}=D_1+\cdots+D_{\dim X}.$}
\end{itemize}
Fix a coordinate $t=(t_1, \cdots, t_{q\dim X})$ of $T$. 
Note that we have a subset $T(t, k_0) \subset T(k)$, 
which is dense in $T$ (cf. Definition~\ref{def-coordinate}). 

We show that there exists $e_0 \in d_0\Z_{>0}$ 
such that the trace map 
$$\Tr_X^{e_0}((p^{e_0}-1)(\Delta+H)+2(D_1+\cdots+D_{\dim X}))$$
is surjective for every $[D_1, \cdots, D_{\dim X}] \in T(t, k_0)$. 
Fix $a=[D_1, \cdots, D_{\dim X}] \in  T(t, k_0)$. 
Then, we can find $d_a \in \Z_{>0}$ such that 
$$\left(X, \Delta+H+\frac{2}{p^{d_a}-1}(D_1+\cdots+D_{\dim X})\right)$$ 
is sharply $F$-pure. 
In particular, we can find $e_a \in d_0d_a\Z_{>0}$ 
such that the trace map 
$$\Tr_X^{e_a}\left((p^{e_a}-1)\left(\Delta+H+\frac{2}{p^{d_a}-1}(D_1+\cdots+D_{\dim X})\right)\right)$$
$$=\Tr^{e_a}_{X\times \{a\}}\left((p^{e_a}-1)\left(\Delta+H+\frac{2}{p^{d_a}-1}\mathcal{D}|_{X\times \{a\}}\right)\right)$$
is surjective. 
By Proposition~\ref{special-general} and the properness of $X$, 
we can find an open subset $a\in U_a \subset T$ such that 
$$\Tr^{e_a}_{X\times \{b\}}\left((p^{e_a}-1)\left(\Delta+H+\frac{2}{p^{d_a}-1}\mathcal{D}|_{X\times \{b\}}\right)\right)$$ 
is surjective for every $\{b\} \in U_a \cap T(t, k_0)$. 
In particular, 
$$\left(X\times \{b\}, \Delta+H+\frac{2}{p^{d_a}-1}\mathcal{D}|_{X\times \{b\}}\right)$$ 
is sharply $F$-pure and 
$$\Tr^{ee_a}_{X\times \{b\}}((p^{ee_a}-1)(\Delta+H)+2\mathcal{D}|_{X\times \{b\}}))$$
is surjective for every $e \in \Z_{>0}$ and every $b \in U_a \cap T(t, k_0)$. 
Therefore, we obtain an open cover 
$$T(t, k_0) \subset \bigcup_{a \in T(t, k_0)} U_a.$$
By the compactness of $T(t, k_0)$, 
we obtain $T(t, k_0) \subset \bigcup_{1\leq i\leq s} U_{a_i}$. 
Set $e_0:=e_{a_1}\cdots e_{a_s}$. 
Then, we see that the trace map  
$$\Tr^{e_0}_{X\times \{c\}}((p^{e_0}-1)(\Delta+H)+2\mathcal{D}|_{X\times \{c\}}))$$
is surjective for every $c \in T(t, k_0)$.


\medskip

Since $|D|$ is base point free and $T(t, k_0)$ is dense in $T$, 
we can find effective Cartier divisors 
$$D_1', \cdots, D_{p^{e_0}-1}' \sim D$$ 
which satisfy the following properties. 
\begin{itemize}
\item{$[D_i'] \in T_1(k)=\mathbb A_k^q(k)$ for every $1\leq i\leq p^{e_0}-1$.}
\item{For every $\{i_1, \cdots, i_{\dim X}\} \subset \{1, 2, \cdots, p^{e_0}-1\}$ 
with $i_1<i_2<\cdots<i_{\dim X}$, 
the point $[D_{i_1}', \cdots, D_{i_{\dim X}}']\in T(k)$ is contained in $T(t, k_0)$.}
\item{$\bigcap_{j\in J} D_j'=\emptyset$ 
for every subset $J \subset \{1, 2, \cdots, p^{e_0}-1\}$ with $|J|=\dim X+1$. }
\end{itemize}
We obtain 
$$D \sim_{\Q} \frac{1}{p^{e_0}-1}(D_1'+\cdots+D_{p^{e_0}-1}')=:D'$$
and 
$$\Tr_X^{e_0}((p^{e_0}-1)(\Delta+H+2D'))$$
is surjective at every point. 
By \cite[Theorem~3.9]{SS}, in order to show that the pair $(X, \Delta+D')$ is strongly $F$-regular, 
it is enough to prove the following two assertions. 
\begin{enumerate}
\item[(a)]{$\Tr_X^{e_0}((p^{e_0}-1)(\Delta+D')+(p^{e_0}-1)H+(p^{e_0}-1)D')$ is surjective.}
\item[(b)]{$(X \setminus \Supp(H+D'), (\Delta+D')|_{X\setminus \Supp(H+D')})$ is globally $F$-regular.}
\end{enumerate}
The assertion (a) follows from  
$$(p^{e_0}-1)(\Delta+D')+(p^{e_0}-1)H+(p^{e_0}-1)D'=(p^{e_0}-1)(\Delta+H+2D').$$
The assertion (b) holds 
because $(X \setminus \Supp H, \Delta|_{X\setminus \Supp H})$ is globally $F$-regular. 
This completes the proof of Proposition~\ref{Theorem-F1}
\end{proof}

The proof of Proposition~\ref{Theorem-F} is 
almost all the same as the one of Proposition~\ref{Theorem-F1}. 
For the sake of completeness, we give a proof of it. 

\begin{proof}[Proof of Proposition~\ref{Theorem-F}]
By enlarging coefficients of $\Delta$, 
we may assume that $\Delta$ is reduced, that is, $\Delta=\llcorner \Delta\lrcorner$. 
By replacing $D$ with its some multiple, 
we may assume that $D$ is a Cartier divisor such that $|D|$ is base point free. 
In particular, $\dim_kH^0(X, D) \geq 2$, otherwise $D \sim 0$ and the assertion is trivial.

Set $T_1:=\mathbb A_k^q=\mathbb A(H^0(X, D))$ (i.e. $q:=\dim_kH^0(X, D)$) and 
$$T:=(T_1)^{\dim X}=\mathbb A_k^q \times_k \cdots \times_k \mathbb A_k^q \simeq \mathbb A_k^{q\dim X}.$$
We can find an effective Cartier divisor $\mathcal D$ on $X \times_k T$ 
which satisfies the following properties: 
\begin{itemize}
\item{Each $k$-rational point $a\in (\mathbb A_k^r(k))^{\dim X} = T(k)$ 
corresponds to a pair $(D_1, \cdots, D_{\dim X})$, 
where, for every $i$, we obtain $D_i \sim D$ is an effective Cartier divisor. 
In this case, we write $a=[D_1, \cdots, D_{\dim X}]$. }
\item{$\mathcal D|_{X\times_k\{[D_1, \cdots, D_{\dim X}]\}}=D_1+\cdots+D_{\dim X}.$}
\end{itemize}
Fix a coordinate $t=(t_1, \cdots, t_{q\dim X})$ of $T$. 
Note that we have a subset $T(t, k_0) \subset T(k)$, 
which is dense in $T$ (cf. Definition~\ref{def-coordinate}).

Since $|D|$ is base point free, 
we can find a non-empty open subset $T^0 \subset T$ 
such that, for every $[D_1, \cdots, D_{\dim X}] \in T(k) \cap T^0$, 
the support $\Supp(\sum D_i)$ does not contain any $F$-pure centers of $(X, \Delta)$. 

We show that there exists $e_0 \in \Z_{>0}$ 
such that, 
for every $[D_1, \cdots, D_{\dim X}] \in T(t, k_0) \cap T^0$, 
the trace map 
$$\Tr_X^{e_0}((p^{e_0}-1)\Delta+(D_1+\cdots+D_{\dim X}))$$
is surjective. 
Fix $a=[D_1, \cdots, D_{\dim X}] \in  T(t, k_0) \cap T^0$. 
Then, by \cite[Main Theorem]{F-adjunction}, 
we can find $d_a \in \Z_{>0}$ such that 
$$\left(X, \Delta+\frac{1}{p^{d_a}-1}(D_1+\cdots+D_{\dim X})\right)$$ is sharply $F$-pure. 
Therefore, we can find $e_a \in d_a\Z_{>0}$ such that 
the trace map 
$$\Tr_X^{e_a}\left((p^{e_a}-1)\left(\Delta+\frac{1}{p^{d_a}-1}(D_1+\cdots+D_{\dim X})\right)\right)$$
$$=\Tr_{X\times_k\{a\}}^{e_a}\left((p^{e_a}-1)\left(\Delta+\frac{1}{p^{d_a}-1}\mathcal{D}|_{X\times_k \{a\}}\right)\right)$$
is surjective at every point. 
By Proposition~\ref{special-general} and the properness of $X$, 
we can find an open subset $a \in U_a \subset T$ such that 
$$\Tr_{X\times_k\{a\}}^{e_a}\left((p^{e_a}-1)\left(\Delta+\frac{1}{p^{d_a}-1}\mathcal{D}|_{X\times_k \{b\}}\right)\right)$$
is surjective for every $b \in T(t, k_0)$. 
Therefore, 
$$\Tr_{X\times_k\{a\}}^{ee_a}\left((p^{ee_a}-1)\left(\Delta+\frac{1}{p^{d_a}-1}\mathcal{D}|_{X\times_k \{b\}}\right)\right)$$
is surjective for every $b \in T(t, k_0)$ and every $e\in \Z_{>0}$. 
By the compactness of $T(t, k_0)$, we can find a required $e_0$. 

\medskip

Since $|D|$ is base point free and $T(t, k_0)$ is dense in $T$, 
we can find effective Cartier divisors 
$$D_1', \cdots, D_{p^{e_0}-1}' \sim D$$ 
which satisfy the following properties. 
\begin{itemize}
\item{$[D_i'] \in T_1(k)=\mathbb A_k^q(k)$ for every $1\leq i\leq p^{e_0}-1$.}
\item{For every $\{i_1, \cdots, i_{\dim X}\} \subset \{1, 2, \cdots, p^{e_0}-1\}$ 
with $i_1<i_2<\cdots<i_{\dim X}$, 
the point $[D_{i_1}', \cdots, D_{i_{\dim X}}']\in T(k)$ is contained in $T(t, k_0) \cap T^0$.}
\item{$\bigcap_{j\in J} D_j'=\emptyset$ 
for every subset $J \subset \{1, 2, \cdots, p^{e_0}-1\}$ with $|J|=\dim X+1$. }
\end{itemize}
We obtain 
$$D \sim_{\Q} \frac{1}{p^{e_0}-1}(D_1'+\cdots+D_{p^{e_0}-1}')=:D'$$
and 
$$\Tr_X^{e_0}((p^{e_0}-1)(\Delta+D'))$$
is surjective. 
Thus, $(X, \Delta+D')$ is sharply $F$-pure. 
\end{proof}

\begin{proof}[Proof of Theorem~\ref{Theorem-MMP}]
First we show the assertion for the case $\K=\Q$. 
Since there exists a log resolution, 
we may assume that $X$ is regular and $\Delta$ is simple normal crossing. 
By Proposition~\ref{Theorem-F} (resp. Proposition~\ref{Theorem-F1}), 
we can find $D' \sim_{\Q} D$ such that $(X, \Delta+D)$ is sharply $F$-pure (resp. strongly $F$-regular). 
In particular, it is log canonical (resp. klt) by \cite[Theorem~3.3]{HW}. 

Second we prove the assertion for the case $\K=\R$. 
We can write $D=\sum_{i=1}^r a_iD_i$ 
where $a_i\in\R_{>0}$ and $D_i$ is a semi-ample Cartier divisor. 
By the induction on $r$, we may assume that $r=1$. 
Thus, we obtain $D=a_1D_1$ where $a_1\in\R_{>0}$ and 
$D_1$ is a semi-ample Cartier divisor. 
By replacing $a_1$ with $\ulcorner a_1\urcorner$, 
we may assume that $D$ is a semi-ample Cartier divisor. 
Since there exists a log resolution, 
we may assume that $X$ is regular and $\Delta$ is simple normal crossing. 
By enlarging coefficients of $\Delta$ a little, we can reduce the problem to the case $\K=\Q$. 
\end{proof}





\end{document}